\documentclass{scrartcl}

\usepackage{amsmath}
\usepackage{amssymb}
\usepackage{todonotes}
\usepackage{amsthm}
\usepackage{mathtools}
\usepackage{enumitem}
\usepackage{csquotes}
\usepackage{tikz-cd}
\usepackage{hyperref}

\DeclareMathOperator{\Lip}{Lip}

\theoremstyle{definition}
\newtheorem{Def}{Definition}
\numberwithin{Def}{section}

\theoremstyle{plain}
\newtheorem{lem}[Def]{Lemma}
\newtheorem{thm}[Def]{Theorem}
\newtheorem{prop}[Def]{Proposition}
\newtheorem{cor}[Def]{Corollary}
\theoremstyle{remark}
\newtheorem{rem}[Def]{Remark}

\title{Some observations on transitivity of Lipschitz operators }
\author{Christian Bargetz \and Leon Kügler}

\begin{document}

\maketitle

\begin{abstract}
\noindent\textbf{Abstract.}
The universal property of the Lipschitz-free spaces allows for the linearisation of a base-point-preserving Lipschitz map $f\colon M\to M$ to a bounded linear operator $T_f\colon \mathcal{F}(M) \to \mathcal{F}(M)$. While it is known that the operator $T_f$ is weakly mixing whenever $f$ has this property, it is open whether a similar inheritance result for topological transitivity holds. Motivated by this question we investigate properties of $T_f$ under the assumption that $f$ is topologically transitive. In particular, we prove Kitai's theorem for Lipschitz operators, i.e. we show that every connected component of the spectrum of the complex variant of $T_f$ intersects the unit circle if $f$ is a topologically transitive Lipschitz map. 

\vspace{3mm}
\end{abstract}

\section{Introduction}

A (discrete) dynamical system is a pair $(M,f)$ where $M$ is a metric space and $f\colon M\to M$ is a continuous mapping. Such a dynamical system is called \emph{topologically transitive} if for every pair of non-empty open sets $U,V\subset M$ there is an $n\in\mathbb{N}$ with $f^n(U)\cap V\neq \emptyset$. If $M$ is a complete, separable metric space without isolated points, by Birkhoff's transitivity theorem the property of being topologically transitive is equivalent to the existence of a point with a dense orbit, i.e., the existence of $x\in M$ for which
\[
  \operatorname{orb}(x,f) = \{x, f(x), f^2(x), \ldots\}
\]
is a dense subset of $M$.

In this note we are interested in a particular class of dynamical systems. We consider pointed metric spaces, i.e. metric spaces in which a distinguished point $0_M\in M$ is chosen. Instead of continuous functions, we restrict our attention to Lipschitz functions $f\colon M\to M$ which fix the base point, i.e., $f(0_M)=0_M$.

For this class of metric spaces, there are Banach spaces which allow for the linearisation of Lipschitz maps. More precisely, the Lipschitz-free space $\mathcal{F}(M)$ on a pointed metric space together with an isometric embedding $\delta_M\colon M \to \mathcal{F}(M)$ can be uniquely characterised by the property that for each Banach space $Y$ and every Lipschitz mapping $g\colon M \to Y$ with $g(0_M)=0$ there is a unique bounded linear operator $T_g\colon \mathcal{F}(M)\to Y$ with $\|T_g\|=\Lip(g)$, where $\Lip(G)$ is the Lipschitz constant of $g$, and such that the diagram
\begin{center}
  \begin{tikzcd}
    M \arrow{r}{g} \arrow[swap]{d}{\delta_M} & Y \arrow{d}{\mathrm{Id}_Y} \\
    \mathcal{F}(M) \arrow{r}{T_g} & Y
  \end{tikzcd}
\end{center}
is commutative, see e.g. Theorem~3.6 in~\cite[p.~84]{Weaver}. These spaces have been introduced by Arens and Eells in~\cite{ArensEells} and received a lot of attention in Banach space theory after the publication of~\cite{GodefroyKalton} by Godefroy and Kalton. Given two pointed metric spaces $M,N$ we may use the isometric embedding $\delta_N\colon N\to\mathcal{F}(N)$ which has the additional property of $\delta_N(0_N)=0$ to extend every Lipschitz mapping $f\colon M\to N$ with $f(0_M)=0_N$ to a Lipschitz mapping $\tilde{f}\colon M\to \mathcal{F}(N)$ with $\tilde{f}(0_M)=0$. Using the above universal property, we now obtain a bounded linear operator $T_f\colon \mathcal{F}(M) \to \mathcal{F}(N)$ for which the diagram
\begin{center}
  \begin{tikzcd}
    M \arrow{r}{f} \arrow[swap]{d}{\delta_M} & N \arrow{d}{\delta_N} \\
    \mathcal{F}(M) \arrow{r}{T_f} & \mathcal{F}(N)
  \end{tikzcd}
\end{center}
is commutative. In~\cite{ACP2023:CompactWeaklyCompactLipschitzOperators} for operators of this form, i.e. bounded linear operators which arise as the linearisation of Lipschitz mappings, the name \emph{Lipschitz operator} has been introduced. In the aforementioned article, compactness and weak compactness of these operators have been studied.

Specialising to the case $M=N$ and a Lipschitz mapping $f\colon M\to M$ which fixes the base point we return to the situation of the dynamical system considered above. Now the linearisation results in the commutative diagram
\begin{center}
  \begin{tikzcd}
    M \arrow{r}{f} \arrow[swap]{d}{\delta_M} & M \arrow{d}{\delta_M} \\
    \mathcal{F}(M) \arrow{r}{T_f} & \mathcal{F}(M)
  \end{tikzcd}
\end{center}
where we might also consider the \emph{linear dynamical system} $(\mathcal{F}(M), T_f)$. In linear dynamics, operators that admit a dense orbit are called \emph{hypercyclic}.

In~\cite{Murillo2015Chaotic} Murillo-Arcila and Peris showed that in certain situations dynamical properties of a bounded linear operator $T$ are inherited from the behaviour on a $T$-invariant subset, or on a sequence of $T$-invariant subsets. In particular, they showed that the stronger properties of weak mixing and mixing and the combination of weak mixing and chaos are inherited from $T$-invariant subsets containing the origin. Since $\delta(M)\subset \mathcal{F}(M)$ is a $T_f$-invariant subset containing the origin for every Lipschitz mapping $f\colon M\to M$ with $f(0_M)=0_M$, they were able to conclude that $T_f$ inherits the properties of weak mixing, mixing and the combination of weak mixing and chaos from $f$. In~\cite{Abbar2021Dynamics} Abbar, Coine and Petitjean systematically investigated the connection between dynamical properties of $f$ and of $T_f$. On the negative side, they gave an example of a Lipschitz mapping $f\colon M\to M$ with $f(0_M)=0_M$ on a discrete metric space which admits a dense orbit but its linearisation $T_f$ does not. While this shows that the property of admitting a dense orbit is not inherited by the linearisation, since on discrete spaces this is not equivalent to topological transitivity, the question of whether topological transitivity is preserved remains open. On the positive side, using splitting theorems for dynamical systems, Abbar, Coine and Petitjean showed that if $f\colon M\to M$ is a base-point-preserving and topologically transitive Lipschitz mapping on a finitely branching compact $\mathbb{R}$-tree, then $T_f$ is also topologically transitive. In fact, they showed that it is even weakly mixing. In~\cite{KüglerLeon2026DpoL} the second author was able to extend this result to compact topological graphs. In~\cite{Abbar2021Dynamics} Abbar, Coine and Petitjean also gave the \emph{Hypercyclicity criterion for Lipschitz operators} and the \emph{Chaoticity criterion for Lipschitz operators}. These criteria are sufficient conditions on $f$ to ensure that $T_f$ is hypercyclic and chaotic, respectively.

After these two publications, the dynamical properties of Lipschitz operators attracted more attention as the recent publications~\cite{Cobollo2025Disjoint, Tapia2024RecurrenceAndVectorsEscaping} show.

Motivated by the question of whether topological transitivity is inherited by $T_f$, we investigate properties of $T_f$ that are implied by $f$ being topologically transitive. More precisely, we show that $T_f$ satisfies a number of necessary conditions for hypercyclicity whenever $f$ is topologically transitive. Maybe most importantly, we prove Kitai's theorem for Lipschitz operators, i.e. we show that every connected component of the spectrum of the complex variant of $T_f$, which we call $S_f$, intersects the unit circle provided $f$ is topologically transitive. 

The first result in Section~\ref{ssec:EigenValues} already appeared in the second author's Master's thesis~\cite{KüglerLeon2026DpoL} which was written under the supervision of the first author.

\section{Preliminaries and Notation}
\subsection{Lipschitz-free spaces}

The space $\mathcal{F}(M)$ can be constructed in the following way. We consider the space
\[
  \Lip_0(M) := \{f\colon M\to \mathbb{K}\colon f\;\text{Lipschitz}, f(0_M)=0\}
\]
of Lipschitz functions vanishing at the base point and equip it with the norm
\[
  \|f\|_{\Lip_0(M)} := \Lip(f).
\]
Note that the condition $f(0)=0$ ensures that the above is indeed a norm. It is easy to see that the point evaluation functionals $\delta_x$ for $x\in M$ define continuous linear functionals on $\Lip_0(M)$. Using this, we consider the mapping
\[
  \delta \colon M \to (\Lip_0(M))^*, \qquad x \mapsto \delta_x
\]
and define the space $\mathcal{F}(M)$ as the closed linear span of $\delta(M)$ in $(\Lip_0(M))^*$ and equip it with the norm induced by the norm on $(\Lip_0(M))^*$. It can be shown that the dual space of $\mathcal{F}(M)$ is $\Lip_0(M)$.

Since in the literature on Lipschitz-free spaces, these spaces are typically considered as real Banach spaces, we will write $\mathcal{F}(M)$ for the real case and $\mathcal{F}(M,\mathbb{C})$ for the complex case. A note of caution seems to be necessary here to point out that not all results on $\mathcal{F}(M)$ in the literature are true in the complex case.

We refer the interested reader to Weaver's book~\cite{Weaver} for a detailed exposition of these spaces and their dual spaces.

\subsection{Linear Dynamics}
We briefly recall the dynamical properties of bounded linear operators, which are relevant for this short note. Let $X$ be a Banach space and $T\colon X\to X$ be a bounded linear operator.

The operator $T$ is called \emph{hypercyclic} if there is an $x\in X$ for which the orbit
\[
  \operatorname{orb}(x,T) = \{T^nx\colon n\in\mathbb{N}_0\}
\]
is a dense subset of $X$. The points with dense orbit are also called \emph{hypercyclic vectors} of $T$. By Birkhoff's transitivity theorem hypercyclicity of $T$ is equivalent to $T$ being topologically transitive, i.e., that for every pair $U,V\subset X$ of non-empty open sets there is an $m\in\mathbb{N}$ with $T^m(U)\cap V \neq \emptyset$. Birkhoff's transitivity theorem also shows that for a hypercyclic operator the set of hypercyclic vectors is a dense $G_\delta$-set.

The operator $T$ is called \emph{weakly mixing} if the operator $T\oplus T \colon X\oplus X\to X\oplus X$ is hypercyclic. Since $T\oplus T$ is quasiconjugate to $T$, this is a stronger notion than being hypercyclic.

The operator $T$ is called \emph{mixing} if for every pair of non-empty open sets $U,V\subset X$ there is an $N\in\mathbb{N}$ such that $T^n(U)\cap V\neq\emptyset$ for all $n\geq N$. This turns out to be a stronger notion than weak mixing.

Finally, the operator $T$ is called \emph{(Devaney) chaotic} if it is hypercyclic and the set
\[
  \operatorname{Per}(T) := \{x\in X\colon \exists n\in\mathbb{N} \;\text{with}\; T^nx=x\}
\]
of periodic points is a dense subset of $X$.

We refer the interested reader to~\cite{LinearChaos} for a detailed presentation of these notions.

\section{Evidence for the inheritance of transitivity}

\subsection{Eigenvalues and orbits of $T_f^*$}\label{ssec:EigenValues}
There are two necessary conditions for hypercyclicity of an operator that have to do with its adjoint, see Proposition~5.1 in \cite{LinearChaos}. We show that transitivity of $f$ implies both of them for the Lipschitz operator $T_f$. We first observe that since the span of $\delta(M)$ is a dense subspace of $\mathcal{F}(M)$, the operator $T_f$ is uniquely determined by its values on the span of $\delta(M)$. Given $\mu = \sum_{k=1}^{m} a_k \delta_{x_k}$ in this span, the operator $T_f$ acts by
\[
    T_f(\mu) = T_f\Big(\sum_{k=1}^{m} a_k \delta_{x_k}\Big) = \sum_{k=1}^{m} a_k \delta_{f(x_k)}.
\]
Given $g\in\Lip_0(M)$, the evaluation of $g$ at $T_f(\mu)$ is given as
\[
    \langle g, T_f(\mu) \rangle = \sum_{k=1}^{m} a_k g(f(x_k)) = \langle g\circ f, \mu\rangle.
\]
This shows that the adjoint of $T_f$ is the composition operator $C_f$.

\begin{prop}
  Let $M$ be a separable complete pointed metric space containing more than one point and $f\colon M\to M$ a topologically transitive Lipschitz mapping fixing the base point. Then the operator 
  \[
  T_f^*=C_f: \Lip_0(M) \to \Lip_0(M), \qquad g \mapsto g \circ f
  \]
  has no eigenvalues. Equivalently, $T_f- \lambda I$ has dense range for any $\lambda \in \mathbb{R}$.
\end{prop}

\begin{proof}
  We first note that a non-singleton metric space with isolated points does not admit a topologically transitive continuous mapping with a fixed point, see Ex.~1.2.3 in~\cite[p.~25]{LinearChaos}. Hence Birkhoff's transitivity theorem allows us to conclude that the set
  \begin{equation*}
    \mathcal{D} \coloneq \{\tilde{x} \in M \hspace{4pt}\vert \hspace{4pt} \text{orb}(\tilde{x}, f) \hspace{4pt} \text{is dense in} \hspace{3pt} M\}.
  \end{equation*}
  is a dense subset of~$M$. With the aim of arriving at a contradiction, we assume that the composition operator $T_f^*=C_f$ has an eigenvalue $\lambda \in \mathbb{R}$. Picking a corresponding eigenvector, we obtain $g \in \Lip_0(M) \setminus \{0\}$ with
  \begin{equation*}
     C_f^n(g) = g \circ f^n= \lambda^n g.
  \end{equation*}
  for all $n\geq 1$. Now we distinguish two cases.

  First we assume that $\vert \lambda \vert \leq 1$. Now, for every point $\tilde{x}\in M$ with a dense orbit, we have
  \begin{align*}
    \underset{x \in M}{\text{sup}} \hspace{2pt} | g(x)| &= \hspace{-4pt}\underset{x \in \text{orb}(\tilde{x},f)}{\text{sup}} | g(x)| \hspace{4pt} = \hspace{4pt} \underset{n \geq 0}{\text{sup}} \hspace{2pt} | g(f^n(\tilde{x}))| \hspace{4pt}= \hspace{3pt} \begin{cases}
      \hspace{2pt} \underset{n \geq 0}{\text{sup}} \hspace{4pt} \vert \lambda^n \vert \vert g(\tilde{x})\vert  & \text{if } \hspace{4pt} \lambda \neq 0 \\[6pt]
      \vert g(\tilde{x}) \vert \hspace{2pt} & \text{if } \hspace{4pt}\lambda = 0
    \end{cases} \\
                                                        & \leq \hspace{3pt} \vert g(\tilde{x})\vert \hspace{1pt}.
  \end{align*}
  In other words, $g$ attains its maximum in every point with dense orbit and hence is constant on the dense set $\mathcal{D}$. The continuity of $g$ and the fact that $g$ vanishes on the base point, imply that $g$ is the zero-function. This is a contradiction to the choice of $g$ as an eigenvector.

  Otherwise, if $\vert \lambda \vert > 1$ we may use that $\mathcal{D}$ is a dense set and $g$ a non-zero Lipschitz function to pick a point $\tilde{x}\in M$ with dense orbit and $g(\tilde{x})\neq 0$. We now choose an arbitrary $\varepsilon>0$ and show that $g$ is unbounded on the ball $B_{\varepsilon}(\tilde{x})$, which is a contradiction to $g$ being Lipschitz. To this end we fix an arbitrary $R > 0$. Because of $\vert \lambda \vert >1$, we can choose a large enough $m \in \mathbb{N}$ with
  \begin{equation*}
    \vert \lambda\vert^m \vert g(\tilde{x})\vert > R.
  \end{equation*}
  Since $M$ does not have isolated points, we may conclude that $M$ has no isolated points. Hence, the orbit of $f^m(\tilde{x})$ is also a dense set. Since $B_{\varepsilon}(\tilde{x})$ is a non-empty open set, this allows us to choose $n \in \mathbb{N}$ such with $f^{m+n}(\tilde{x}) \in B_{\varepsilon}(\tilde{x})$. For this point, we obtain
  \begin{equation*}
    \vert g(f^{m+n}(\tilde{x}))\vert = \vert (C_f^{m+n}(g))(\tilde{x}) \vert =  \vert \lambda \vert^{m+n} \vert g(\tilde{x})\vert   > \vert \lambda \vert^{m} \vert g(\tilde{x})\vert > R.
  \end{equation*}
  from the above. Since $R>0$ was chosen arbitrarily, we obtain that $g$ is unbounded on $B_\varepsilon(\tilde{x})$.
  
  The fact that the absence of eigenvalues of the adjoint is equivalent to $T_f-\lambda I$ having dense range for any $\lambda \in \mathbb{R}$ can be found e.g. in Lemma 2.53 in~\cite[p.~52]{LinearChaos}.
\end{proof}

\begin{rem} \label{ComplexCaseEVofAdjoint}
  The above proof also works for the complex version of the Lipschitz operator, which we denote by $S_f$. Also in this case, the absence of eigenvalues of the adjoint is equivalent to $S_f-\lambda I$ having dense range for any $\lambda \in \mathbb{C}$.
\end{rem}

\begin{prop}
  Let $M$ be a separable complete pointed metric space containing more than one point and $f\colon M\to M$ a topologically transitive Lipschitz mapping fixing the base point. Then, every function in $\Lip_0(M)\setminus\{0\}$ has an unbounded orbit under the composition operator  $T_f^*=C_f\colon \Lip_0(M) \to \Lip_0(M)$.
\end{prop}

\begin{proof}
  If $M$ is a singleton, then the statement is vacuously true since  $\Lip_0(M)\setminus\{0\}=\emptyset$. So we can assume that $M$ has more than one point.
  To produce a contradiction, we assume that there are $R>0$ and $g \in \Lip_0(M) \setminus \{0\}$ such that
  \begin{equation*}
    \sup_{n\geq 0} \|C_f^n(g)\|_{\Lip_0(M)} = \sup_{n \geq 0} \Lip(g \circ f^n) \leq R .
  \end{equation*}
  We now show that $g$ has to be the zero function. In order to do so, let $y \in M\setminus \{0\}$ be arbitrary. We show that $g(y)=0$.  
  To shorten the notation we set $L:=\Lip(g)$.  Given an $\varepsilon>0$, we consider the non-empty, open sets $B_{\frac{\varepsilon}{R}}(0)$ and $B_{\frac{\varepsilon}{L}}(y)$. The topological transitivity of $f$ allows us to pick a point $\tilde{x}\in B_{\frac{\varepsilon}{R}}(0)$ and an $n > 0$ with $f^n(\tilde{x}) \in B_{\frac{\varepsilon}{L}}(y)$. We obtain
  \begin{align*}
    |g(y)| &\leq |g(y)-g(f^n(\tilde{x}))|+|g(f^n(\tilde{x})| \leq L d(y,f^n(\tilde{x})) + \Lip(g \circ f^n) \, d(\tilde{x},0) \\ & \leq 
             \varepsilon + R  \ d(\tilde{x},0) < \varepsilon + R \frac{\varepsilon}{R}=2 \varepsilon
  \end{align*}
  Since $y$ and $\varepsilon$ where arbitrary, we get the contradictory statement $g \equiv 0$
\end{proof}

\subsection{Kitai's theorem for Lipschitz operators}

For hypercyclic linear operators on complex Banach spaces Kitai's Theorem, see~\cite{Kitai} or Theorem~5.6 in~\cite[p.~140]{LinearChaos} states that every connected component of the spectrum has to intersect the unit circle $\mathbb{T}\subset \mathbb{C}$. We now show that this also true for complex Lipschitz operators coming from a topologically transitive Lipschitz map.

For the proof, we use the following two well-known results from spectral theory and a topological Lemma which we reproduce here for the convenience of the reader.

\begin{lem}[{Lemma~5.2 in~\cite[p.~138]{LinearChaos}}] \label{ExponentialBound}
  Let $T$ be a bounded operator on a complex Banach space $X$ and $r>0$.
  \begin{itemize}
  \item[(i)] If $\sigma(T) \subset \{z \in \mathbf{C} : |z|<r\}$, then there are $\varepsilon \in (0,r)$ and $A >0$ such that
    \begin{equation*}
      \|T^nx\| \leq A (r- \varepsilon)^n \|x \|
    \end{equation*}
    for all $x \in X$ and $n\geq 0$.
  \item[(ii)] If $\sigma(T) \subset \{z \in \mathbf{C} : |z|>r\}$, then there are $\varepsilon \in (0,r)$ and $A >0$ such that
    \begin{equation*}
      \|T^nx\| \geq A (r+ \varepsilon)^n \|x \|
    \end{equation*}
    for all $x \in X$ and $n\geq 0$.
  \end{itemize}
\end{lem}

\begin{prop}[Riesz decomposition theorem, {Theorem~B.9 in~\cite[p.~364]{LinearChaos}}]
  \label{prop:RieszDecomposition}
  Let $T$ be a bounded operator on a complex Banach space $X$. Suppose the spectrum $\sigma(T)$ can be partitioned into two nonempty disjoint closed subsets  $\sigma_1$ and $\sigma_2$. Then there are nontrivial $T$-invariant closed subspaces $X_1$ and $X_2$ such that 
  \begin{equation*}
    X = X_1 \oplus X_2 \ , \ \ \sigma(T|_{X_1}) = \sigma_1 \ , \ \  \sigma(T|_{X_2}) = \sigma_2.
  \end{equation*}
  Note that since $\sigma(T)$ is a closed subset of $\mathbb{C}$, a subset of $\sigma(T)$ is closed in $\sigma(T)$ if and only if it is closed in $\mathbb{C}$.
\end{prop}

\begin{lem}[{Lemma~1.21 in~\cite[p.~11]{bayart2009dynamics}}] \label{Lem:ClopenSuperset}
Let $\sigma$ be a compact subset of $\mathbb{C}$, and let $C$ be a connected component of $\sigma$. Assume that $C$ is contained in some open set $\Omega\subset \mathbb{C}$. Then there exists a subset $\sigma_1 \subset \sigma$ that is clopen in $\sigma$ and satisfies $C \subseteq\sigma_1 \subseteq \Omega$.
\end{lem}

Recall that in order to distinguish the real from the complex case, we denote by $S_f$ the complex Lipschitz operator of $f$ that is defined on the complex Lipschitz-free space $\mathcal{F}(M, \mathbb{C})$.

With this preparation at hand, we are now able to state and prove Kitai's Theorem for Lipschitz operators.

\begin{thm} \label{Thm:KitaiForLipschitzOpsGeneral}
  Let $M$ be a separable complete pointed metric space, containing more than one point and $f\colon M\to M$ a topologically transitive Lipschitz mapping fixing the base point of $M$. Then every connected component of the spectrum of $S_f$ intersects the unit circle.
\end{thm}

\begin{proof}
We argue by contradiction and assume that the spectrum $\sigma(S_f)$ has a connected component $C$ that does not intersect the unit circle. Therefore, $C$ either lies entirely in the open unit disk $\mathbb{D}$ or the open set $\mathbb{C} \setminus \overline{\mathbb{D}}$. Lemma~\ref{Lem:ClopenSuperset} allows us to find a clopen subset $\sigma_1$ of $\sigma(S_f)$ that contains $C$ and also lies entirely in  $\mathbb{D}$ or $\mathbb{C} \setminus \overline{\mathbb{D}}$, respectively. If we have the equality $\sigma_1=\sigma(S_f)$, we set $P \coloneqq \mathrm{id}_{\mathcal{F}(M, \mathbb{C})}$,  $X_1 \coloneqq \mathcal{F}(M, \mathbb{C})$ and $T \coloneqq S_f$. Otherwise, the clopenness of $\sigma_1$ in $\sigma(S_f)$  implies that $\sigma_1$ and $\sigma(S_f) \setminus \sigma_1$ form a partition of $\sigma(S_f)$ into two nonempty closed subsets of $\sigma(S_f)$. The Riesz decomposition theorem implies the existence of nontrivial, $S_f$-invariant, complemented and closed subspaces $X_1, X_2 \subset \mathcal{F}(M, \mathbb{C})$ with $\mathcal{F}(M,\mathbb{C})=X_1\oplus X_2$ and $\sigma(S_f \big \vert _ {X_1})=\sigma_1$. In this case we denote by $P$ the corresponding projection onto $X_1$ and set $T \coloneqq S_f \big \vert_{X_1}$.

In both cases we continue with the same arguments to arrive at the contradiction. We consider the mapping
  \begin{equation*}
    \phi \colon M \to X_1, \qquad x \mapsto P(\delta_x).
  \end{equation*}

Because a non-singleton metric space with isolated points does not admit a topologically transitive continuous mapping with a fixed point, see Ex.~1.2.3 in~\cite[p.~25]{LinearChaos},  $M$~cannot have isolated points. Therefore, Birkhoff's transitivity theorem allows us to pick a point $\tilde{x}\in M$ with dense orbit under~$f$. Since the subspaces $X_1$ and $X_2$ are invariant under $S_f$, the operators $P$ and $S_f$ commute and we have
  \begin{equation} \label{PhiMapsOrbit}
    \phi(f^n(x))=P(\delta_{f^n(x)})=P(S_f^n(\delta_x))=S_f^n(P(\delta_x))=T^n(\phi(x))
  \end{equation}
for all $n \geq0$ and all $x \in M$. Since the set $\sigma_1=\sigma(T)$ lies either entirely inside the open unit disk or entirely in the complement of the closed unit disk, Lemma~\ref{ExponentialBound} implies that either all orbits under $T$ converge to $0$ or all orbits of nonzero vectors under $T$ go to infinity. In particular, the vector $\phi(\tilde{x})$ must exhibit one of these two behaviors. Let us first derive a contradiction from behavior number one, i.e. the case where
\begin{equation*}
    \underset{n \to \infty}{\lim} T^n(\phi(\tilde{x}))=0.
\end{equation*}
Fix an arbitrary point $y \in M$. Because the orbit of $\tilde{x}$ under $f$ is dense and $M$ does not have isolated points, there exists an increasing sequence of integers $(n_k)_{k \in \mathbb{N}}$ such that
\begin{equation*}
    \underset{k \to \infty}{\lim} f^{n_k}(\tilde{x})=y.
\end{equation*}
By continuity of $\phi$ we have
\begin{equation*}
0=\underset{n \to \infty}{\lim} T^n(\phi(\tilde{x}))= \underset{k \to \infty}{\lim} T^{n_k}(\phi(\tilde{x}))\overset{\eqref{PhiMapsOrbit}}{=}\underset{k \to \infty}{\lim} \phi( f^{n_k}(\tilde{x}))=\phi(y).
\end{equation*}
Because $y$ is an arbitrary point in $M$, we have shown that $\phi=P \circ \delta$ vanishes on $M$.  Since the set $\{\delta_x\colon x\in M\}$ is a total subset of $\mathcal{F}(M,\mathbb{C})$ this implies that $P \equiv 0$ and thus $X_1=\{0\}$. From the Riesz decomposition theorem we obtained that $X_1$ must be non-trivial. Hence, we found a contradiction.\\
Now we consider behavior number two, i.e. the case
\begin{equation*}
    \underset{n \to \infty}{\lim} \| T^n(\phi(\tilde{x}))\| =\infty.
\end{equation*}
Here we use the facts that $\mathrm{orb}(\tilde{x},f)$ is dense in $M$ and that $M$ does not have isolated points to choose a sequence $(n_k)_{k \in \mathbb{N}}$ such that
\begin{equation*}
    \underset{k \to \infty}{\lim} f^{n_k}(\tilde{x})=0.
\end{equation*}
The continuity of $\| \phi(\cdot)\|$ allows us to conclude that
\begin{equation*}
    \infty =\underset{n \to \infty}{\lim} \| T^n(\phi(\tilde{x}))\| = \underset{k \to \infty}{\lim} \| T^{n_k}(\phi(\tilde{x}))\| \overset{\eqref{PhiMapsOrbit}}{=} \underset{k \to \infty}{\lim} \| \phi(f^{n_k}(\tilde{x})) \|= \| \phi(0) \| =0
\end{equation*}
which is a contradictory statement. 
\end{proof}

\subsection{Power-compactness, power-boundedness and paranormality}
In what follows we look at three classes of operators on Banach spaces that can never be hypercyclic. We show that a transitive $f$ cannot induce a Lipschitz operator with either one of these properties.

\begin{Def}[Power-compactness, power-boundedness, paranormality]
  An operator $T$ on a Banach space is called \textit{power-compact} if some iterate $T^n$ is compact. $T$ is called \textit{power-bounded} if \[\underset{n\geq 0}{\mathrm{sup}}\|T_f^n\|< \infty.\] $T$ is called \textit{paranormal} if for all $x \in X$ it satisfies \[\|Tx\|^2 \leq \|T^2x\|\|x\|.\]
\end{Def}

We start with the class of power-compact operators. In Theorem~A in~\cite[p.~1004]{ACP2023:CompactWeaklyCompactLipschitzOperators} compactness of the Lipschitz operator $T_f$ is characterised in terms of properties of the Lipschitz mapping $f\colon M\to M$. In particular, the authors show that for $T_f$ to be compact it is necessary that $f$ is uniformly locally flat, i.e. it has to satisfy the condition
\begin{equation} \label{UniformLocalFlatness}
  \lim_{\delta \to 0^+} \sup \Big \{\frac{d(f^n(x), f^n(y))}{d(x,y)} \ \ \Big \vert \ \ x,y \in M, \ 0<d(x,y)< \delta \Big \}=0.
\end{equation}
We now show that no iterate of a topologically transitive base-point-preserving Lipschitz mapping can have this property.

\begin{lem} \label{Lemma:NoIterateUniformlyLocFlat}
  Let $M$ be a separable complete pointed metric space, containing more than one point and $f\colon M\to M$ a topologically transitive Lipschitz mapping fixing the base point of $M$. Then no iterate of $f$ is uniformly locally flat.
\end{lem}

\begin{proof}
  We assume that $f^n$ is uniformly locally flat and argue towards a contradiction. We set $L \coloneqq \Lip(f)$ and now distinguish two cases.

  For the first case, we assume that $L\leq 1$. Because $M$ is not a singleton, we can choose an $\varepsilon>0$ such that the disjoint open sets $B_{\varepsilon}(0)$ and $M \setminus \overline{B_{\varepsilon}(0)}$ are both nonempty. Now the conditions $L\leq 1$ and $f(0)=0$ imply that $B_{\varepsilon}(0)$ is an $f$-invariant set, contradicting topological transitivity of $f$.

  For the second case, we assume $L>1$. By the uniform local flatness, see~\eqref{UniformLocalFlatness}, there is a $\delta>0$ with 
  \begin{equation*}
    d(f^n(x),f^n(y) \leq \frac{1}{L^n} d(x,y).
  \end{equation*}
  for all $x,y\in M$ with $d(x,y)<\delta$. Note that we may choose $\delta$  small enough so that the open sets $U \coloneqq B_{\delta / L^n}(0)$ and $V \coloneqq M \setminus \overline{B_{\delta}(0)}$ are nonempty whereas the condition $L>1$ ensures that the sets are disjoint. Since every $x \in U$ satisfies
  \begin{equation*}
    d(f^n(x),0) = d(f^n(x), f^n(0)) \leq \frac{1}{L^n} d(x,0)< \frac{\delta}{L^{2n}}<\frac{\delta}{L^n},
  \end{equation*}
  the set~$U$ is invariant under~$f$. Now we fix an arbitrary $m \in \mathbb{N}_0$ and represent it as $m=nk+r$, with $k \in \mathbb{N}_0$ and $r \in \{0, \dots , n-1\}$. Every $x \in U$ satisfies
  \begin{equation*}
    d(f^m(x),0)=d(f^r(f^{nk}(x)),f^r(0))\leq L^r d(f^{nk}(x),0)<L^r \frac{\delta}{L^n} \leq \delta \ ,
  \end{equation*}
  where the penultimate inequality uses the fact that $U$ is $f^n$-invariant. We have shown that $f^m(U) \cap V = \emptyset$. Since $m$ was arbitrary, this contradicts the transitivity of~$f$.
\end{proof}

\begin{cor}
  Let $M$ be a separable complete pointed metric space, containing more than one point and $f\colon M\to M$ a topologically transitive Lipschitz mapping fixing the base point of $M$. Then $T_f$ is not power-compact.
\end{cor}

\begin{proof}
  By Theorem A in~\cite[p.~1004]{ACP2023:CompactWeaklyCompactLipschitzOperators}, if an iterate $T_f^n=T_{f^n}$ is compact, then $f^n$ is uniformly locally flat. This is impossible by Lemma~\ref{Lemma:NoIterateUniformlyLocFlat}.
\end{proof}

Next we show that no topologically transitive~$f$ can have a power-bounded linearisation~$T_f$.

\begin{prop}
  Let $M$ be a separable complete pointed metric space, containing more than one point and $f\colon M\to M$ a topologically transitive Lipschitz mapping fixing the base point of $M$. Then $T_f$ is not power-bounded.
\end{prop}

\begin{proof}
  Power-boundedness of $T_f$ implies that
  \begin{equation*}
    \sup_{n \geq0}\|T_f^n\|=\sup_{n \geq0}\|T_{f^n}\|= \sup_{n \geq0} \Lip(f^n) < \infty
  \end{equation*}
  We define $s \coloneqq \max \{1, \sup_{n \geq0} \Lip(f^n)\}$ and use that $M$ is not a singleton to be able to pick a $\delta>0$ for which the open sets $U\coloneqq B_{\delta}(0)$ and $ V \coloneqq M \setminus \overline{B_{s\delta}}$ are nonempty. Note that $s\geq1$ ensures that the sets are disjoint. Now, for $x \in U$ and $n \geq 0$ we have
  \begin{equation*}
    d(f^n(x),0)= d(f^n(x),f^n(0)) \leq s d(x,0)<s \delta
  \end{equation*}
  and therefore $f^n(x) \notin V$ for all $n \geq 0$. Hence $f$ cannot be topologically transitive.
\end{proof}

\begin{prop}
  Let $M$ be a separable complete pointed metric space, containing more than one point and $f\colon M\to M$ a topologically transitive Lipschitz mapping fixing the base point of $M$. Then $T_f$ is not paranormal.
\end{prop}

\begin{proof}
  As observed before, the existence of a transitive base-point-preserving Lipschitz map on the non-singleton space $M$ implies that $M$ cannot have isolated points. Hence Birkhoff's transitivity theorem allows us to choose a point $\tilde{x} \in M$ with dense orbit under $f$.

  Now suppose $T_f$ is paranormal. Exactly as in the proof of Theorem 5.30 in~\cite[p.~153]{LinearChaos}, the paranormality condition implies that any orbit is either decreasing in norm or eventually strictly increasing. Clearly, the orbit of $\delta_{\tilde{x}}$ can have neither of these properties, because it must visit every nonempty open subset of $\delta(M)\subset\mathcal{F}(M)$ infinitely often.
\end{proof}

\section{An observation about Devaney Chaos}
If $M$ is separable and without isolated points, then the combination of transitivity and the existence of a dense set of periodic points is known as \textit{Devaney chaos}. As for topological transitivity, it is not known whether the property of Devaney chaos is inherited from $f$ to $T_f$ without the additional assumption of weak mixing. The existence of a dense set of periodic points is a property that is passed from $f$ to $T_f$. Therefore, general inheritance of transitivity would also imply general inheritance of Devaney chaos. Proposition~5.7 in \cite[p.~140]{LinearChaos} says that among the eigenvalues of a Devaney chaotic operator on a complex Banach space there are infinitely many roots of unity. Moreover, the spectrum has no isolated points.

Inspired by this statement, we show the following two propositions that can be viewed as evidence for the general transfer of Devaney chaos from $f$ to $T_f$, at least if the underlying space $M$ is connected.

\begin{prop} \label{fDevaneyPointSpecRootsOfUnity}
  Let $M$ be separable without isolated points and $f$ Devaney chaotic. Then the point spectrum of $S_f$ contains infinitely many roots of unity.
\end{prop}

\begin{proof}
  Because $S_f$ is an extension of $f$, $S_f$ has a periodic point. In fact, it is easily checked that $S_f$ inherits the fact that periodic points are dense from $f$. In the real case, this is Proposition 2.6 in~\cite[p.~45]{Abbar2021Dynamics} and the proof also works for the complex case. Now we use the existence of a periodic point together with the fact that
  \begin{equation} \label{PeriodicPointsFormula}
    \mathrm{Per}(S_f)= \mathrm{span}\{\mu \in \mathcal{F}(M, \mathbb{C}) \ \vert \ S_f \mu =e^{\alpha \pi i }\mu \;\text{for some}\; \alpha\in\mathbb{Q}\}\ ,
  \end{equation}
  which is Proposition~2.33 in \cite[p.~43]{LinearChaos}. We assume that the point spectrum contains exactly $N$ roots of unity, denoted by $\lambda_1, \cdots, \lambda_N$, and argue towards a contradiction. This implies that 
  \begin{equation*}
    p(S_f)=(S_f- \lambda_1 I)\dots (S_f- \lambda_N I)
  \end{equation*}
  maps any eigenvector to 0. Note that the factors in the expression commute. From~\eqref{PeriodicPointsFormula} we may conclude that this operator vanishes on the dense set of periodic points. Therefore, the operator is the zero operator. This is a contradiction to Remark~\ref{ComplexCaseEVofAdjoint}, since every factor in p(T) has dense range and thus p(T) must have dense range.
\end{proof}

\begin{prop} \label{fDevaneyFullSpecIsoPoints}
  Let $M$ be a connected separable complete pointed metric space, containing more than one point and $f\colon M\to M$ a Devaney chaotic Lipschitz mapping fixing the base point of $M$.
  Then the spectrum of $S_f$ has no isolated points.
\end{prop}

\begin{proof}
  Suppose $\lambda \in \mathbb{C}$ is an isolated point of $\sigma(S_f)$. By the Riesz decomposition theorem, see Proposition~\ref{prop:RieszDecomposition}, there are complemented, nontrivial, $S_f$-invariant closed subspaces $X_1, X_2\subseteq \mathcal{F}(M, \mathbb{C})$ with $\sigma(S_f \big \vert _{X_1})= \{\lambda\}$ and $X=X_1\oplus X_2$. We denote by $P$ the corresponding projection onto $X_1$, set $T \coloneqq S_f \big \vert _{X_1}$ and consider the Lipschitz mapping
  \begin{equation*}
    \phi\colon M \to X_1, \qquad x \mapsto P(\delta_x).
  \end{equation*}
  Since $X_1$ is $S_f$-invariant, $P$ and $S_f$ commute and we have
  \begin{equation} \label{PhiCompatibleWithf}
    \phi(f^n(x))=P(\delta_{f^n(x)})=P(S_f^n(\delta_x))=S_f^n(P(\delta_x))=T^n(\phi(x))
  \end{equation}
  for all $n \geq0$ and all $x \in M$. We now make a case distinction and deduce the contradictory statement $X_1 =\{0\}$ in both cases.

  In the first case, we assume that $\lambda$ is not a root of unity, i.e. $\lambda^k \neq 1$ for every integer $k \geq 1$. By the spectral mapping theorem, we have $\sigma(T^k)= \{\lambda^k\}$ and thus $1 \notin \sigma(T^k)$ for any $k \geq 1$, which implies that $T^k-I$ is invertible. Let $p \in \mathrm{Per}(f)$ be any periodic point of $f$ and denote its period by $k$. We have
  \begin{equation*}
    \phi(p)=\phi(f^k(p))=T^k(\phi(p))
  \end{equation*}
  and thus
  \begin{equation*}
    (T^k-I)\phi(p)=0.
  \end{equation*}
  Since $T^k-I$ is injective, we obtain $\phi(p)=0$. Since $p$ was an arbitrary periodic point, the continuity of $\phi$ togehter with the fact that the periodic points of $f$ are dense im $M$, we may conclude that $\phi(x)=0$ for all $x\in M$. Since $\{\delta_x\colon x\in M\}$ is a total subset of the space $\mathcal{F}(M)$, this implies that $P=0$ and hence $X_1=\{0\}$.

  In the second case, we assume that $\lambda$ is a root of unity and denote by $m\geq1$ the smallest positive integer with $\lambda^m=1$. Thus $\sigma(T^m-I)=\{0\}$ by the spectral mapping theorem. Put differently, $T^m= I+N$ is a quasi-nilpotent perturbation of the identity. Now let $p \in \mathrm{Per}(f)$ and choose $k \geq 2$ such that $f^k(p)=p$. Then
  \begin{equation*}
    \phi(p)=\phi(f^{km}(p))=T^{km}(\phi(p))=(I + N)^k\phi(p)=0.
  \end{equation*}
  and thus
  \begin{equation} \label{FormulawithN}
    ((I+N)^k-I) \phi(p)=0.
  \end{equation}
  The binomial theorem implies the polynomial identity
  \begin{equation*}
    (1+t)^k-1=kt(1+q(t))
  \end{equation*}
  where 
  \begin{equation*}
    q(t)\coloneqq \sum_{i=2}^{k} \binom{k}{i} \frac{1}{k} t^{i-1}.  
  \end{equation*}  
  After plugging $N$ into the identity and substituting this into~\eqref{FormulawithN}, we obtain
  \begin{equation*}
    kN(I+q(N))\phi(p)=k(I+q(N))N\phi(p)=0.
  \end{equation*}
  Since $q$ is a polynomial with no constant term and $N$ is quasinilpotent, we have
  \[
    \sigma(I+q(N))=\{1\}
  \]
  by the spectral mapping theorem and hence the operator $I+ q(N)$ is invertible. This forces $N\phi(p)=0$. Because $p$ was an arbitrary periodic point of $f$ and the periodic points of $f$ are dense in $M$, the operator N vanishes on $\mathrm{im}(\phi) $. By linearity and boundedness it vanishes everywhere. We have deduced
  \begin{equation} \label{IterateIsIdentity}
    T^m=I+N=I.
  \end{equation}
  We use Birkhoff's transitivity theorem to pick a point $\tilde{x}\in M$ with a dense orbit under~$f$. Since $\phi$ is continuous, it follows that also $\phi(M)$ is connected. The set $\phi(\mathrm{orb(\tilde{x},f)})$ is dense in $\phi(M)$. By Equation~\eqref{PhiCompatibleWithf}, this is exactly the orbit of $\phi(\tilde{x})$ under $T$. But Equation~\eqref{IterateIsIdentity} tells us that this orbit consists of at most $m$ points. So it must actually be equal to $\phi(M)$ and by connectedness of $\phi(M)$ it consists of a single point. There is no other option than this point being the origin 0 and $\phi$ being the zero mapping. Again, like in the previous case, this means that the projection $P$ vanishes on the evaluation functionals and thus everywhere.
\end{proof}

\vspace{5mm}
\noindent
Christian Bargetz\\
Universität Innsbruck\\
Department of Mathematics\\
Technikerstraße 13, 6020 Innsbruck, Austria\\
\texttt{christian.bargetz@uibk.ac.at}

\vspace{5mm}
\noindent
Leon Kügler\\
Universität Innsbruck\\
\texttt{leon.kuegler.math@gmail.com}

\end{document}